\documentclass[11pt]{amsart}
\usepackage{amsmath}
\usepackage{fullpage}
\usepackage[norelsize, lined, boxed]{algorithm2e}

\title{On the number of matroids}
\author{N. Bansal}
\address{Eindhoven University of Technology, Eindhoven, the Netherlands}
\thanks{This research has been supported
by  the Netherlands Organisation for Scientific Research (NWO) grant 639.022.211.}
\email{bansal@gmail.com}

\author{R. A. Pendavingh}
\address{Eindhoven University of Technology, Eindhoven, the Netherlands}
\email{R.A.Pendavingh@tue.nl}

\author{J. G. van der Pol}
\address{Eindhoven University of Technology, Eindhoven, the Netherlands}
\email{j.g.v.d.pol@tue.nl}
\date{}
\begin{document}
\newcommand{\CC}{\mathcal{C}}
\newcommand{\BB}{\mathcal{B}}
\newcommand{\II}{\mathcal{I}}
\newcommand{\FF}{\mathcal{F}}
\newcommand{\MM}{\mathbb{M}}
\newcommand{\ZZ}{\mathcal{Z}}
\newcommand{\R}{\mathbb{R}}
\newcommand{\N}{\mathbb{N}}
\newcommand{\cl}{\mbox{cl}}
\newcommand{\e}{\text{e}}
\newcommand{\ind}[1]{\II\left(#1\right)}
\newtheorem{theorem}{Theorem}
\newtheorem{lemma}[theorem]{Lemma}
\newtheorem{corollary}[theorem]{Corollary}
\newtheorem{conjecture}[theorem]{Conjecture}
\newtheorem{definition}[theorem]{Definition}
\newtheorem*{remark*}{Remark}
\newcommand{\ignore}[1]{}
\newcommand{\defeq}{:=}		%{\stackrel{\text{def}}{=}}
\newcommand{\eqdef}{:=}		%{\stackrel{\text{def}}{=}}
\newcommand{\eps}{\varepsilon}
\makeatletter
\newcommand{\statetheorem}[2]{\expandafter\def\csname thmcontent@#1\endcsname {#2}\begin{theorem}\label{#1}#2\end{theorem}\newtheorem*{thm:#1}{Theorem \ref{#1}}}
\newcommand{\repeattheorem}[1]{\begin{thm:#1}\csname thmcontent@#1\endcsname \end{thm:#1}}
\newcommand{\statecorollary}[2]{\expandafter\def\csname crlcontent@#1\endcsname {#2}\begin{corollary}\label{#1}#2\end{corollary}\newtheorem*{crl:#1}{Corollary \ref{#1}}}
\newcommand{\repeatcorollary}[1]{\begin{crl:#1}\csname crlcontent@#1\endcsname \end{crl:#1}}
\makeatother

\setlength{\parskip}{0.2ex}

% Caption of algorithm now starts with "Procedure #:"
\renewcommand{\algorithmcfname}{Figure}

\begin{abstract}
We consider the problem of determining $m_n$, the number of matroids on $n$ elements.
The best known lower bound on $m_n$ is due to Knuth (1974) who showed that $\log \log m_n$ is at least $n-\frac{3}{2}\log n-O(1)$. On the other hand, Piff (1973) showed that  $\log\log m_n\leq n-\log n+\log\log n +O(1)$, and it has been conjectured since that the right answer is perhaps closer to Knuth's bound.

We show that this is indeed the case, and prove an upper bound on  $\log\log m_n$ that is within an additive $1+o(1)$ term of Knuth's lower bound.
Our proof is based on using some structural properties of non-bases in a matroid together with some properties of stable sets in the Johnson graph to give a compressed representation of matroids.
\end{abstract}

\maketitle

\section{Introduction}
Matroids, introduced
by Whitney in his seminal paper \cite{Whitney1935}, are fundamental combinatorial objects
and have been extensively studied due to their very close connection to combinatorial optimization, see e.g.\ ~\cite{SchrijverBookB}, and their ability to abstract core notions from areas such as graph theory and linear algebra \cite{Kung1996, OxleyBook}.

There are several ways to define a matroid. Perhaps the most natural one is using the notion of independence.
A matroid $M$ is a pair $(E,\II)$, where $E$ is the ground set of elements, and $\II$ is a nonempty collection of subsets of $E$ called the independent sets with the following properties:
\begin{enumerate}
\item Subset property: $A \in \II$ implies $A'\in \II$ for all $A'\subset A$, and
\item Exchange property: If $A,B \in \II$ with  $|A|>|B|$, then there exists an element $x$ in $A\setminus B$, such that $B \cup \{x\} \in \II$.
\end{enumerate}

A basic question is: how many distinct matroids are there on a ground set of $n$ elements?
We denote this number by $m_n$.
Clearly, there are $2^n$ subsets of $E$ and hence at most $2^{2^n}$ ways to choose $\II$, which gives the trivial upper bound $\log \log m_n \leq n$. Here, and throughout the paper, $\log$ denotes the logarithm to the base 2.

This bound is easily improved to $\log \log m_n \leq n - \frac{1}{2}\log n + O(1)$ by focussing on matroids of a fixed {\em rank}. In a matroid, the maximal independent sets are called {\em bases}, and by the exchange property all bases of a matroid have the same cardinality. This common cardinality is the rank of the matroid. Let $m_{n,r}$ be the number of matroids of rank $r$.
As $m_n= m_{n,0}+m_{n,1}+\ldots + m_{n,n}$, it must hold that  $m_{n,r} \geq m_n/(n+1)$ for some $r$.
By the subset property, any matroid of rank $r$ is completely determined by specifying its bases. As there are at most $\binom{n}{r} \leq \binom{n}{\lfloor n/2 \rfloor} =O(2^n/\sqrt{n})$ (call this $\ell$) such bases, this gives $m_{n,r} \leq  2^\ell$ and thus
$$\log \log m_n \leq  \log \log ((n+1) m_{n,r}) \leq \log \log ((n+1) 2^\ell)  = n - \frac{1}{2} \log n + O(1).$$

%\subsection{The best known bounds.}
In 1973, Piff \cite{Piff1973} improved this bound further to  $\log\log m_n\leq n-\log n+\log\log n +O(1)$, by observing that a matroid is also completely determined by the closures of its circuits, and using a counting argument to show that there ``only" $O(2^n/n)$ such closures
%using a more compact representation of matroids by exploiting the properties of so-called cyclic flats
(we describe Piff's proof in Section \ref{ss:prelim_piff}). This is the best upper bound known to date.

In the other direction, the best known lower bound is due to Knuth \cite{Knuth1974} from 1974, who showed that
$\log\log m_n \geq n - \frac{3}{2}\log n - O(1)$.
Knuth's bound is based on an elegant construction of  matroids whose {\em non-bases}\protect\footnote{For a matroid of rank $r$, a non-basis is an $r$-subset of the ground set that is dependent.}
% I.e.~they form the complement of the bases in $\binom{E}{r}$.}
satisfy a particular property.
Specifically, he constructs a large family of so-called {\em sparse paving matroids}. These are matroids of rank $r$, where any two non-bases %of size $r$
intersect in at most $r-2$ elements (i.e.\ ~their incidence vectors have Hamming distance $4$ or more). Such sets of non-bases are precisely the stable sets\protect\footnote{We use {\em stable set} as a synonym for what is often called {\em independent set} in graph theory (i.e.\ ~a set of vertices, no two of which are adjacent). As independent set has a different meaning in matroid theory, this serves to avoid confusion.} in the so-called {\em Johnson graph} $J(n,r)$. This is the graph with vertex set %$\binom{[n]}{r}$
$\{X\subseteq[n]\mid |X|=r\}$, in which two vertices are adjacent if and only if their intersection contains $r-1$ elements.

Knuth's bound follows by taking a collection of $k=\frac{1}{n} \binom{n}{\lfloor n/2 \rfloor}$ such non-bases ---~equivalently, a stable set in the graph $J(n,n/2)$ of this size, of which an explicit description is provided in Section \ref{ss:prelim_knuth}~--- %(section \ref{ss:prelim_knuth} has an explicit description of this set)
and considering the family of size $2^k$ of sparse paving matroids obtained by taking each possible subset of this family. Thus $m_n \geq s_n \geq 2^k$, where $s_n$ is the number of sparse paving matroids on $n$ elements. This gives the lower bound
\begin{equation}
\label{knuth:lb}
\log \log m_n \geq \log \log s_n \geq  \log k =  n - \frac{3}{2}\log n  + \frac{1}{2} \log \frac{2}{\pi} - o(1).
\end{equation}
We explain Knuth's bound in more detail in Section \ref{ss:prelim_knuth}.

Historically, the interest in paving matroids seems to be a response to the publication of the catalog of matroids on at most 8 elements by Blackburn, Crapo, and Higgs \cite{BlackburnCrapoHiggs1973} in the early 1970's. With reference to such numerical evidence, Crapo and Rota consider it probable that paving matroids ``would actually predominate in any asymptotic enumeration of geometries'' \cite[p.~3.17]{CrapoRotaBook}.
In his book ``Matroid Theory'', Welsh also notes that paving matroids predominate among the small matroids, and puts the question whether this pattern extends to matroids in general as an exercise \cite[p.~41]{WelshBook}.  An earlier lower bound on the number of matroids due to Piff and Welsh \cite{PiffWelsh1971} was also based on a bound on the number of (sparse) paving matroids. Mayhew and Royle recently confirmed that the predominance of sparse paving matroids extends to the matroids on 9 elements \cite{MayhewRoyle2008}.

In recent years, (sparse) paving matroids have received attention in relation to a wide variety of matroid topics \cite{Jerrum2006, GeelenHumphries2006, MerinoNoble2010, Bonin2011}. These authors all suggest that  the class of sparse paving matroids is probably a very substantial subset of all matroids, pointing out Knuth's argument for the lower bound.

Mayhew, Newman, Welsh and Whittle \cite{MayhewNewmanWelshWhittle2011} present a very nice collection of conjectures on the asymptotic behavior of matroids.  In particular, they conjecture that asymptotically almost every matroid is sparse paving:
\begin{conjecture}[Mayhew, Newman, Welsh and Whittle \cite{MayhewNewmanWelshWhittle2011}]\label{MNWW}
	$\lim_{n \rightarrow \infty} s_n/m_n =1$.
\end{conjecture}
If true, this would imply:
\begin{conjecture}\label{main:conj}
$\log \log m_n = \log \log s_n + o(1)$.
\end{conjecture}
Note that this is in fact a much weaker statement as $\log \log (\cdot)$ is a very ``forgiving'' function, e.g.~ if $m_n = \Omega(ns_n)$ or even if $m_n = \Omega(2^{2^{\sqrt{n}}}s_n)$, then $\frac{s_n}{m_n} \to 0$, while still $\log \log m_n = \log \log s_n + o(1)$.

\subsection{Our results}

Our main result is a substantial strengthening of the upper bound on $m_n$. %Specifically, we show that:

\statetheorem{thm:mn} {The number of matroids $m_n$ on $n$ elements satisfies $$\log \log m_n \le n - \frac{3}{2} \log n + \frac{1}{2} \log \frac{2}{\pi} + 1 + o(1).$$}
Combining Theorem \ref{thm:mn} with Knuth's lower bound \eqref{knuth:lb} on the number of sparse paving matroids $s_n$, this gives:

\statecorollary{relative_bound}{$\log \log m_n \le \log\log s_n + 1 + o(1)$.}

Thus, this result comes quite close to Conjecture \ref{main:conj}, except for the additive $+1$ term. In particular,
it implies that the number of matroids is indeed much closer to Knuth's lower bound, and perhaps also lends support to the conjecture that most matroids are indeed sparse paving.

%In this paper, we substantially improve Piff's upper bound on $m_n$. We also extend Knuth's argument to give an improved upper bound on %$s_n$. \\

\subsection{Our Techniques}

The proof of Theorem \ref{thm:mn} is based on a combination of the following:
\begin{enumerate}
\item A technique for proving refined upper bounds on the total number of stable sets in a graph.
\item Defining a notion of {\em local cover} of a matroid, which serves as a short certificate to identify the bases in the
neighborhood of an $r$-set. Combining the local covers for a carefully chosen set of $r$-sets then serves as a compressed representation of
any matroid.
\end{enumerate}

To see the connection to the total number of stable sets, note that any upper bound on $m_n$ is also an upper bound on $s_n$. As $s_n= s_{n,0} + s_{n,1} + \ldots + s_{n,n}$, where $s_{n,r}$ denotes the number of sparse paving matroids of rank $r$, and $s_{n,r}$ is precisely the total number of stable sets in the Johnson graph $J(n,r)$, any method to upper bound $m_n$ must also bound the number of such sets.

We first give an overview of each of these two ideas, and then describe how they are combined to prove Theorem \ref{thm:mn}.
These ideas are already useful by themselves to improve the currently known bounds on $s_n$ and $m_n$.
In Section \ref{s:bound-weak} we show how local covers can be used in a very simple way to obtain the following bound.
\statetheorem{thm:mn_loglog}{$\log \log m_n \le n - \frac{3}{2} \log n + 2\log\log n + O(1)$.}
While this bound is weaker than the one in Theorem \ref{thm:mn}, it already improves Piff's upper bound substantially, and matches Knuth's lower bound up to the additive $O(\log \log n)$ term.

Similarly, in Section \ref{s:sparsepaving} we show how the refined counting technique for stable sets implies the following bound on the number of sparse paving matroids.

\statetheorem{thm:sn}{$\log \log s_n \le n - \frac{3}{2} \log n + \frac{1}{2} \log \frac{2}{\pi} +1+ o(1)$.}
Previously, the best known upper bound on $s_n$ seems to be $\log \log s_n  \le n - \frac{3}{2} \log n + O(\log \log n)$ \cite{MayhewWelsh2010} (we sketch an argument below).

Finally, we prove Theorem \ref{thm:mn} in Section \ref{s:bound}.

\subsubsection{Upper-bounding $m_n$ via local covers:}
Recall that $m_{n,r}$ denotes the number of matroids of rank $r$ on $n$ elements. As $m_n = m_{n,0} + m_{n,1} + \ldots + m_{n,n}$, it suffices to bound each $m_{n,r}$ separately.
 For a matroid of rank $r$, let us call a collection of {\em flats} a {\em flat cover} if it completely describes the matroid by certifying for each $r$-set whether it is a basis or not.

  A related notion is that of a {\em local cover}: a collection of flats that allows us to identify the bases in the neighborhood of some fixed $r$-set.
Our main observation is that given any matroid, for every $r$-set, one can associate to it a local cover consisting of at most $r$ flats.
This implies that if we pick any dominating set $D$ in the Johnson graph and list all the local covers for the vertices in $D$, then this gives a valid flat cover consisting of at most $|D|r$ flats.
Together with standard arguments about the existence of small dominating sets in any regular graph, this implies
that each matroid $M \in \MM_{n,r}$ can be described by a ``small'' flat cover, which gives the bound in Theorem \ref{thm:mn_loglog}.

\subsubsection{Upper-bounding $s_n$ via stable sets:}
As $s_n = s_{n,0} + s_{n,1} + \ldots + s_{n,n}$ it suffices to bound each of these terms separately.
For a graph $G$,
let $i(G)$ denote the number of stable sets in $G$, and recall that $s_{n,r} = i(J(n,r))$.
While it is hard to obtain any reasonable estimate of $i(G)$ for general graphs, it was shown in \cite{MayhewWelsh2010} that
\begin{equation}
\label{triv-sn-count}
\log \log s_n \leq n - \frac{3}{2} \log n +  \log \log n + O(1)
\end{equation}
One may argue this as follows.
Let $G=(V,E)$ be a $d$-regular graph and let $-\lambda$ denote the smallest eigenvalue of its adjacency matrix. Then
the size of a maximum stable set of $G$ is at most $|V|\lambda/(d+\lambda)$ by Hoffmann's bound (see e.g. Theorem 3.5.2 of \cite{BrouwerHaemers2012}, or our Corollary \ref{cor:degree}). Let us denote $\alpha = \lambda/(d+\lambda)$. This implies that
\begin{equation}
\label{eq:triv-count}
i(G) \leq \sum_{j=0}^{\alpha|V|} \binom{|V|}{j}.
\end{equation}

For the graph $J(n,r)$ it is known that $\alpha$ is at most  $2/n$, which implies that the maximum stable set has size at most $\alpha N$ where $N= \binom{n}{r}$.
Note that this bound is quite good and is within a factor $2$ of the size of the explicit stable set used in Knuth's lower bound.
Applying~\eqref{eq:triv-count} to $J(n,r)$ then gives $s_{n,r} \leq 2^{(1+o(1))2N\log n/n}$, which implies the bound~\eqref{triv-sn-count}.
We note that the proof of~\eqref{triv-sn-count} in~\cite{MayhewWelsh2010} is similar, except that there the same bound on the maximal size of a stable set  of $J(n,r)$ was shown by a combinatorial argument.

It turns out however that counting all the subsets in \eqref{eq:triv-count} is rather wasteful and that this
bound can be improved. In particular, we show:
\statetheorem{thm:indsets}{Let $G$ be a $d$-regular graph on $N$ vertices with smallest eigenvalue $-\lambda$, with $d>0$. Then
	$i(G)\leq \sum_{i=0}^{\lceil\sigma N\rceil} {N \choose i}2^{\alpha N},$
where $\alpha= \frac{\lambda}{d + \lambda}$ and $\sigma=\frac{\ln(d+1)}{d+\lambda}$.}

For the graph $J(n,r)$, we find that $\sigma N \leq \frac{8 \ln n}{n^2}N$  and hence  this gives the stronger bound $i(G) \leq 2^{(2+o(1))N/n}$.
As $\alpha N = 2N/n$ was our bound on the size of the maximum stable set, this bound on $i(G)$ roughly implies that most stable sets occur as subsets of a few large stable sets of size $\alpha N$. Using standard bounds on the binomial coefficients, this directly implies Theorem \ref{thm:sn}.

%Our proof of Theorem \ref{thm:indsets} is very closely based on an approach due to  Alon, Balogh, Morris, and Samotij \cite{ABMS2012} to bound $i(G)$ in a regular graph based on the smallest eigenvalue $\lambda$. To do this, they give a procedure to encode any independent set $I$ using few bits.
%While we use the same encoding procedure as in \cite{ABMS2012}, we do a slightly more refined analysis (specifically, Lemma \ref{lemma:Sbound}), as the bounds of \cite{ABMS2012} are not so useful for the range of parameters that are of interest to us.

Our proof of Theorem \ref{thm:indsets} is based on a procedure for encoding stable sets that is originally due to Kleitman and Winston \cite{KleitmanWinston1982}. We remain very close to the description of the procedure as given in Alon, Balogh, Morris and Samotij \cite{ABMS2012}, to which we also refer for detailed references on the earlier uses of the procedure. Compared to \cite{ABMS2012}, we give a somewhat improved analysis (specifically, Lemma \ref{lemma:Sbound}) to obtain a sufficient bound in the parameter range that is of interest to us.

\subsubsection{The improved upper bound on $m_n$:}
To obtain the bound in Theorem \ref{thm:mn}, we combine the two ideas above.
The main observation is that given a matroid $M$,  if $X$ is a dependent $r$-set (i.e.~a non-basis) in $M$, then $X$ has a local cover consisting of at most $2$ flats
(as opposed to up to $r$ flats if $X$ was an arbitrary $r$-set).
Thus if we could construct a flat cover using few such local covers, then we would obtain a much smaller description of a matroid.
To this end, we generalize the procedure of Alon et al.\ ~\cite{ABMS2012} for encoding stable sets to more generally encode
flat covers of the kind described above using a few number of bits. This gives the improved bound on $m_{n,r}$ and hence on $m_n$.

%Turning once more to the encoding procedure used in the proof of Theorem \ref{thm:indsets}, we encode the set of non-bases $K$ of a matroid $M \in \MM_{n,r}$ and obtain sets $S$ and $A$ such that $$S \subseteq K \subseteq S \cup N(S) \cup A,$$ where $S$ is an independent set in $J(n,r)$ and the sizes of $S$ and $A$ can be controlled. We show that $S \cup N(S)$ can be described by a collection of $|S|$ local covers, each of which contains at most 2 flats. As $A$ is small we can use this to bound $m_{n,r}$ and hence $m_n$. This approach improves the bound in Theorem \ref{thm:mn_loglog}.

Finally, we remark that the $+1$ additive gap in our upper bound on $m_n$ arises only because of the factor $2+o(1)$ gap between the known upper and lower bounds on the size of the maximum stable set in the graphs $J(n,r)$ for $r\approx n/2$. It is likely that reducing this gap could lead to improved bounds for $m_n$. In Section \ref{s:outtro}, we elaborate on this issue a bit further.

\section{Preliminaries}
\subsection{Matroids}
As mentioned previously, a matroid $M$ is specified by $M=(E,\II)$, where the sets in the collection $\II$ satisfy the independence axioms.
The elements of $\II$ are {\em independent}, the remaining elements of $2^E\setminus \II$ are {\em dependent}. The set $E$ is the {\em ground set}, and we say that $M$ is a matroid {\em on} $E$. There are various set systems and functions defined on $M$ that each allow one to distinguish between dependent and independent sets, such as the set of bases, the rank function, the circuits, the closure operator, etc. We define these notions and state some of their basic properties here, but for a detailed account of their interrelations and for proofs we refer to Oxley \cite{OxleyBook}.

A {\em basis} of $M$ is an inclusionwise maximal independent set of $M$. It follows from the independence axioms, that each basis has the same cardinality.
 In this paper, we will present matroids as  $M=(E,\BB)$, where $\BB$ is the set of bases of $M$. The following is an alternate characterization of matroids in terms of the basis axioms, which we shall need later. A set $\BB\subseteq 2^E$ is the set of bases of a matroid on $E$ if and only if $\BB\neq \emptyset$ and $\BB$ satisfies the {\em basis exchange axiom}
\begin{equation}\label{base_exchange}\text{for each }B, B'\in \BB\text{ and each }e\in B\setminus B'\text{ there exists an }f\in B'\setminus B\text{ such that } B-e+f\in\BB.\end{equation}
Here, we write $X+y \defeq X\cup\{y\}$ and $X-y \defeq X\setminus\{y\}$.

The {\em rank} of a set $X\subseteq E$ is $r_M(X)\eqdef\max\{|I|\mid I\subseteq X, ~I\in\II\}$, i.e.~the cardinality of any maximal independent set in $X$. The rank function is {\em submodular}:
$$r_M(X\cap Y)+r_M(X\cup Y)\leq r_M(X)+r_M(Y).$$
We write $r(M)\eqdef r_M(E)$. Then $r(M)$ is the common cardinality of all bases, the {\em rank of $M$}.
We say that an $r$-set $X$ is a {\em non-basis} if $r_M(X)< r$.
Clearly, a matroid of rank $r$ with set of bases $\BB$ is also uniquely defined by its set of {\em non-bases}, $\binom{E}{r} \setminus \BB$.

A {\em circuit} of $M$ is an inclusionwise minimal dependent set of $M$. We denote the set of circuits of $M$ by $\CC(M)$. By definition, each dependent set contains some circuit.
We will use  that if $X$ is an $r$-set with $r_M(X)=r(M)-1$, then it contains a unique circuit $C\subseteq X$.

In $M$, the {\em closure} of a set $X\subseteq E$ is the set $\cl_M(X)\eqdef\{e\in E\mid r_M(X+e)=r_M(X)\}$.
We will often use that $r_M(\cl_M(X))=r_M(X)$ for any set $X$, which follows easily from induction and the submodularity of the rank function.
% the  that we will often use is that
A set $F\subseteq E$ is called a {\em flat} of $M$ if $\cl_M(F)=F$, and $\mathcal{F}(M)$ denotes the set of all flats of $M$.
As $\cl_M(\cl_M(X)) = \cl_M(X)$ for any set $X$, every closure $\cl_M(X)$ is a flat.

The following simple property of flats will be crucially used in our construction of flat covers: A set $X\subseteq E$ is dependent if and only if there exists a flat $F$ such that
$|X \cap F| > r_M(F)$.
In other words, $F$ acts as witness that $X$ contains a dependency when restricted to $F$.

The {\em dual } of $M$ is the matroid $M^*$ whose bases are $\BB^*=\{E\setminus B\mid B\in \BB\}$. The bases, circuits, rank, and closure of sets in $M^*$ are called the cobases, cocircuits, corank, and coclosure of sets in $M$, and we write $r^*_M(X)\eqdef r_{M^*}(X)$, $\CC^*(M)\eqdef\CC(M^*)$, $\cl^*_M\eqdef\cl_{M^*}$, etc.
%For a set $X\subseteq E$ of cardinality $r(M)$, we thus have
%$$X\text{ is dependent }\Longleftrightarrow E\setminus X\text{ is codependent }\Longleftrightarrow E\setminus X\text{ contains a cocircuit }D\in\CC^*(M)$$
%So to certify that  a set $X\subseteq E$ with $|X|=r(M)$ is not a basis, we can either point out an appropriate circuit, or a cocircuit.

The rank and corank functions of $M$ are related by
\begin{equation}\label{dualrank}r_M^*(X)=r_M(E\setminus X)-r(M)+|X|.\end{equation}

%Finally, we say that a flat $F$ of $M$ is {\em cyclic} if and only if $E\setminus F$ is a flat of $M^*$, and note that if $C$ is a circuit, then $\cl_M(C)$ is a cyclic flat.

We write
$$\mathbb{M}_{n}\eqdef\{M\text{ a matroid}\mid E(M)=\{1,\ldots, n\}\}, \quad \mathbb{M}_{n,r}\eqdef\{M\in \mathbb{M}_n\mid~r(M)=r\}.$$
Also, we put
$m_n\eqdef|\mathbb{M}_n|,\quad m_{n,r}\eqdef|\mathbb{M}_{n,r}|.$

A matroid $M$ is {\em paving} if $|C|\geq r(M)$ for each circuit $C$ of $M$ (or equivalently if there is no dependent set of size $<r(M)$), and {\em sparse} if $M^*$ is paving. $M$ is said to be {\em sparse paving} if it is both sparse and paving. We write
$$s_n\eqdef|\{M\in\mathbb{M}_n\mid M\text{ is sparse paving}\}|,\quad s_{n,r}\eqdef|\{M\in\mathbb{M}_{n,r}\mid M\text{ is sparse paving}\}|.$$

\subsection{Bounds on binomial coefficients}
We will frequently use the following standard bounds.
\begin{equation}\label{binomial_bound}\binom{n}{r}\leq \left(\frac{\e n}{r}\right)^r\end{equation}
\begin{equation}\label{binomial_bound2}\frac{2^n}{\sqrt{n}}\left(\sqrt{\frac{2}{\pi}}-o(1)\right)\leq \binom{n}{\lfloor \frac{n}{2}\rfloor}\leq \frac{2^n}{\sqrt{n}}\sqrt{\frac{2}{\pi}}\end{equation}
We will also use the following bound on the sum of binomial coefficients for $k<n/2$:
\begin{equation}\label{binomial_bound3}\sum_{i=0}^k \binom{n}{i} \le \frac{n-(k-1)}{n-(2k-1)} \binom{n}{k}.\end{equation}
 This follows by upper bounding $\binom{n}{k-j}\leq \binom{n}{k} \left(\frac{k}{n-k+1}\right)^j$, and summing up the resulting geometric series. We are particularly interested in the case $k = o(n)$, which yields
%In particular if $k=o(n)$, then
\begin{equation}\label{binomial_bound3-1}\sum_{i=0}^k \binom{n}{i} \le (1+o(1)) \binom{n}{k}.\end{equation}

\subsection{The Johnson graph}
If $E$ is a finite set and $r \le |E|$, then we write $$\binom{E}{r} \defeq \{X \subseteq E \mid |X| = r\}$$ for the collection of $r$-subsets of $E$. We say that $X, Y \in \binom{E}{r}$ are adjacent (notation: $X \sim Y$) if they have Hamming distance $|X \triangle Y| = 2$ (or equivalently: $|X \cap Y| = r-1$).
The {\em Johnson graph} $J(E,r)$ is defined as the graph with vertex set $\binom{E}{r}$, in which two vertices $X$ and $Y$ are adjacent if and only if $X \sim Y$.
We abbreviate $J(n,r) \defeq J([n],r)$.
For any $r$-set $X \in \binom{E}{r}$, we write $N(X) \defeq \{Y \in \binom{E}{r} \mid X \sim Y\}$ for the neighborhood of $X$ in $J(E,r)$.
Obviously, $J(E,r) \cong J(n,r)$ for any $n$-set $E$. \\

The following lemma points out the connection between the Johnson graph and sparse paving matroids. It was essentially shown by Piff and Welsh \cite{PiffWelsh1971} in proving an earlier lower bound on $s_n$.

\begin{lemma}\label{lemma:sparsepaving_johnson}For $0 < r < n$, sparse paving matroids $M \in \MM_{n,r}$ correspond one-to-one to stable sets in $J(n,r)$.\end{lemma}
\proof
Let $E=[n]$. We first show that the non-bases of a rank-$r$ sparse paving matroid $M$ on $E$ form a stable set in $J(E,r)$. Suppose that there are non-bases $X, Y \in \binom{E}{r} \setminus \BB(M)$ such that $X \sim Y$, then we would have
$$r_M(X\cap Y)+r_M(X\cup Y)\leq r_M(X)+r_M(Y)<2r-1,$$
so that either $r_M(X\cap Y)<r-1=|X\cap Y|$ or $r_M(X\cup Y)<r$. In the former case, $X\cap Y$ is a dependent set of size $<r(M)$, which contradicts that $M$ is paving. In the latter case, it follows from \eqref{dualrank} that
$$r_M^*(E\setminus (X\cup Y)) = r_M(X\cup Y) - r(M) + |E\setminus(X\cup Y) | < r -r + |E\setminus(X\cup Y)| = n-r-1 = r^*(M)-1,$$ so that $E\setminus (X\cup Y)$ is a dependent set of $M^*$ of size $<r(M^*)$, which contradicts that $M^*$ is paving. \\

Next, suppose that $I$ is a stable set in $J(E,r)$. We will show that $\BB \defeq \binom{E}{r} \setminus I$ forms a valid collection of bases for some matroid on $E$.

First, it cannot be that $\BB=\emptyset$ as this would imply that $I = \binom{E}{r}$ and hence that $J(E,r)$ has no edges.
 So the only way $\BB$ may fail to be a family of bases is if it fails the basis exchange axiom \eqref{base_exchange}. That is, there are distinct $B,B'\in \BB$ and an $e\in B\setminus B'$ such that $B-e+f\not \in \BB$ for all $f\in B'\setminus B$.
  Now, it must be the case that $|B'\setminus B|>1$, for otherwise it would hold that $B-e+f=B'\in\BB$ for the only $f\in B'\setminus B$.
  So, let $f, f'$ be distinct elements of $B'\setminus B$, and consider $N=B-e+f$ and $N'=B-e+f'$. Since the base exchange axiom fails, it follows that both $N,N'\in I$. On the other hand, $|N\triangle N'|=|\{f,f'\}|=2$, i.e.\ ~$N\sim N'$, contradicting independence of $I$.
\endproof

\subsection{\label{ss:prelim_knuth}Knuth's lower bound}
In \cite{Knuth1974}, Knuth argues that if $J(n,r)$ has a stable set $I$ of size $k$, then $J(n,r)$ has at least $2^k$ stable sets, as each subset of $I$ is itself stable. Knuth constructed a stable set of size $k = \frac{1}{2n}\binom{n}{r}$, but Theorem 1 in \cite{GrahamSloane1980} shows the existence of a stable set of size at least $k = \frac{1}{n}\binom{n}{r}$.

We sketch the construction in \cite{GrahamSloane1980}. Identifying the vertices of $J(n,r)$ with their incidence vectors, we view them as $\{0,1\}$ vectors $(x_1,\ldots,x_n)$ with exactly $r$ 1's. It is easily verified that the functional $\{0,1\}^n \to \mathbb{Z}/n\mathbb{Z}$, defined by $$(x_1, x_2, \ldots, x_n) \mapsto \sum_{i=1}^n i x_i \mod n$$ gives a valid $n$-vertex-coloring of $J(n,r)$. As there are $n$ color classes, at least one of them should contain at least $\frac{1}{n}{n \choose r}$ vertices.

Picking such a stable set gives
$\log(s_{n,r})\geq \frac{1}{n}\binom{n}{r}$, and in particular $\log(s_n)\geq \log (s_{n,\lfloor n/2\rfloor})\geq \frac{2^n}{n\sqrt{n}}\left(\sqrt{\frac{2}{\pi}}-o(1)\right)$ by \eqref{binomial_bound2}. Therefore,
\begin{equation}\label{eq:knuthlowerbound}\log\log s_n\geq n-\frac{3}{2} \log n +\frac{1}{2}\log\frac{2}{\pi}-o(1).\end{equation}

\subsection{\label{ss:prelim_piff}Piff's upper bound}
To prove his upper bound on $m_n$, Piff \cite{Piff1973} uses that any matroid $M$ is characterized by the set of all closures of circuits and their ranks, i.e.\ ~by the collection
\begin{equation}\label{piff}\mathcal{K}(M)\defeq\{(\cl_M(C), r_M(C)) \mid C\text{ a circuit of }M\}.\end{equation}
This completely defines $M$ as a set $X\subseteq E(M)$ is dependent in $M$ if and only if $|X\cap\cl_M(C)|>r_M(C)$ for some circuit $C$ of $M$.
He then uses the following counting argument to bound the size of $\mathcal{K}(M)$.
\begin{lemma}If $M\in\mathbb{M}_n$, then $|\mathcal{K}(M)|\leq \frac{1}{n+1}2^{n+1}$.
\end{lemma}
\proof Fix an $i<n$. Let $C\in \CC(M)$ be a circuit such that $|C|=i+1$. Then for each $e\in C$, we have $\cl_M(C)=\cl_M(C-e)$, i.e. there are $i+1$ sets $C-e\in\binom{E}{i}$ that map to $\cl_M(C)$. It follows that
$$|\{(\cl_M(C),r)\in \mathcal{K}(M)\mid r=i\}|\leq \frac{1}{i+1}\binom{n}{i}=\frac{1}{n+1}\binom{n+1}{i}.$$
Summing these upper bounds over all $i$, we get
\begin{equation*}|\mathcal{K}(M)|=\sum_{i=0}^{n-1}|\{(\cl_M(C),r)\in \mathcal{K}(M)\mid r=i\}|\leq \sum_{i=0}^{n-1}\frac{1}{n+1}\binom{n+1}{i}\leq \frac{1}{n+1}2^{n+1}.\qed\end{equation*}

It follows that the number of matroids on a set $E$ of $n$ elements is at most the number of subsets $\mathcal{K}\subseteq 2^{E}\times\{0,\ldots, n\}$ with $|\mathcal{K}|\leq \frac{1}{n+1}2^{n+1}$. Thus, together by  \eqref{binomial_bound3-1}
$$m_n \leq \sum_{i=0}^{2^{n+1}/(n+1)} \binom{2^n(n+1)}{i}  = (1+o(1))\binom{2^n(n+1)}{\frac{1}{n+1}2^{n+1}} .$$

By \eqref{binomial_bound}, this gives
$\log m_n \leq \frac{2^{n+1}}{n+1}\log \frac{\e(n+1)^2}{2} + o(1)$
%$$\log m_n\leq \log\left((1+o(1))\binom{2^n(n+1)}{\frac{1}{n+1}2^{n+1}}\right)\leq \frac{2^{n+1}}{n+1}\log \frac{\e(n+1)^2}{2} + o(1)$$
and hence
$\log\log m_n\leq n-\log n+\log\log n+O(1).$
\qed

\section{\label{s:bound-weak}A weaker upper bound on the number of matroids}

In this section, we introduce the notion of flat covers and local covers and use them to show that each matroid in $\MM_{n,r}$ has a concise description. Using this, we then bound $m_{n,r}$.
\begin{definition}[Flat cover]
Let $M=(E,\BB)$ be a matroid with $n$ elements, of rank $r$.  For a set $X\subseteq E$, we say that a flat $F\in \FF(M)$ {\em covers} $X$ if $|F\cap X|>r_M(F)$.
We say that a set of flats $\mathcal{Z}$ is a {\em flat cover} of $M$ if  each non-basis $X\in\binom{E}{r}\setminus\BB$ is covered by some $F\in\mathcal{Z}$.
\end{definition}
Note that  if $\mathcal{Z}$ covers $M$, then  $M$ is characterized by $E,r$ and the collection $$\{(F,r_M(F))\mid F\in \mathcal{Z}\},$$
since  by definition of a cover, we have $\BB=\{X\in \binom{E}{r}\mid |X\cap F|\leq r_M(F)\text{ for all }F\in \ZZ\}$.

\begin{definition}[Local cover]
For an $r$-set  $X \in \binom{E}{r}$, we say that a collection of flats $\mathcal{Z}_X \subseteq \FF(M)$ is a {\em local cover} at $X$ if $\mathcal{Z}_X$ covers all the non-bases $Y \in (N(X) \cup \{X\})\setminus \BB$. \\
\end{definition}
\ignore{The set of closures of circuits $\{\cl_M(C)\mid C\in\CC(M)\}$ that  Piff uses to characterize $M$ is also a cover of the nonbases of size at most $2^{n+1}/(n+1)$. In this section, we show that smaller covers exist.}
\begin{lemma} \label{local_cover} Let $M\in\MM_{n,r}$. For each $r$-set $X \in  \binom{E}{r}$, there is a local cover $\mathcal{Z}_X$ such that $|\mathcal{Z}_X|\leq r$.
\end{lemma}
\proof
Let $X$ be a fixed $r$-set. Take $\mathcal{Z}_X\eqdef\{\cl_M(X-x)\mid x\in X\}.$ Then clearly $|\mathcal{Z}_X|\leq r$. We consider a $Y\in N(X)\cup\{X\}$.

If $Y=X$ and $X$ is dependent, then $X\subseteq\cl_M(X-x_0)$ for some $x_0\in X$. Then $\cl_M(X-x_0)$ covers $X$, as
$$|X\cap\cl_M(X-x_0)|=|X|=r>r_M(X) \geq r_M(X-x_0)= r_M(\cl_M(X-x_0)).$$

If $Y\in N(X)$, then $Y=X-x+y$ for some $x\in X$ and $y\in E\setminus X$. If $\cl_M(X-x)$ covers $Y$, we are done. Otherwise,
$$r-1=|X-x|\leq |\cl_M(X-x)\cap Y|\leq r_M(\cl_M(X-x))\leq r-1$$
so that equality holds throughout, and in particular $r_M(X-x)=r-1$ and $y\not\in\cl_M(X-x)$. It follows that $r_M(Y)=r_M(X-x+y)=r$, so that $Y$ is a basis and it is not required to cover $Y$.
\endproof
If $G=(V,E)$ is a graph, then a set $D \subseteq V$ is {\em dominating} if $D\cup N(D)=V$.
The point of introducing local covers is that one can construct a small flat cover from a collection of local covers at the vertices in some small dominating set, as every non-basis in the matroid will be covered by this collection.
By standard probabilistic arguments (see Theorem 1.2.2 of \cite{AlonSpencerBook}),
% or using a result of Lov\'asz \cite{Lovasz1975},
 one has:
\begin{lemma}\label{domset} $J(n,r)$ has a dominating set of cardinality $\frac{\ln(r(n-r)+1)+1}{r(n-r)+1}\binom{n}{r}$.\end{lemma}

\begin{corollary} \label{crl:cover} Let $M\in\MM_{n,r}$. Then $M$ has a flat cover $\mathcal{Z}$ with
$|\ZZ|\leq r\frac{\ln(r(n-r)+1)+1}{r(n-r)+1}\binom{n}{r}$.
\end{corollary}
\proof By Lemma \ref{domset}, $J(n,r)$ has a dominating set $D$  with $|D|\leq \frac{\ln(r(n-r)+1)+1}{r(n-r)+1}\binom{n}{r}$. For each $X\in D$, let $\mathcal{Z}_X$ be a local cover of $M$ at $X$ as in Lemma \ref{local_cover}. Then $|\mathcal{Z}_X|\leq r$ for each $X\in D$.  Take $\mathcal{Z}\defeq\bigcup_{X\in D} \mathcal{Z}_X.$
Then $\mathcal{Z}$ is a cover of $M$, and $|\mathcal{Z}|\leq r|D|$.\endproof

%\begin{theorem} $\log\log m_n\leq n-\frac{3}{2}\log n+ 2 \log\log n + O(1)$.\end{theorem}
\repeattheorem{thm:mn_loglog}
\begin{proof} Denote the upper bound in Corollary \ref{crl:cover} by $k_{n,r}\eqdef r\frac{\ln(r(n-r)+1)+1}{r(n-r)+1}\binom{n}{r}$.
As each matroid in $\MM_{n,r}$ is uniquely determined by the set
$\{(F,r_M(F))\mid F\in\mathcal{Z}\}\subseteq 2^E\times \{0,\ldots, n\},$
where $\mathcal{Z}$ is a cover of size bounded by $k_{n,r}$, the number of matroids in $\MM_{n,r}$ is bounded by the number of subsets of
a set of size $2^n(n+1)$ of cardinality at most $k_{n,r}$.

Now, for sufficiently large $n$,
$$\max_{r\leq n/2} k_{n,r} \leq \max_{r\leq n/2}\frac{\ln(n^2)+1}{n-r} \binom{n}{r} \leq \frac{2+4 \ln n}{n}\binom{n}{\lfloor n/2\rfloor}.$$
By \eqref{binomial_bound2} this is at most $O(2^n (\ln n) /n^{3/2})$ and hence by \eqref{binomial_bound3-1} and \eqref{binomial_bound},
$$ \max_{r\leq n/2} m_{n,r} \le (1+o(1)) \binom{2^n(n+1)}{O(2^n (\ln n)/n^{3/2})} \leq (O(n^{5/2}))^{O(2^n (\ln n)/n^{3/2})} $$
The same bound applies to  $\max_{r\geq n/2} m_{n,r}$, since $m_{n,r} = m_{n,n-r}$ by matroid duality. Thus,
$$\max_{r\leq n}\log m_{n,r}\leq    O \left( \frac{2^n \ln n}{n^{3/2}}\right) \cdot O(\log n),$$ and hence
$\log \log m_n \leq n - \frac{3}{2} \log n + 2\log\log n + O(1)$
as required.
%$$\log m_{n,r}\leq k_{n,\lfloor n/2\rfloor}\log \frac{\e2^n(n+1)}{k_{n,\lfloor n/2\rfloor}} \leq \frac{4 \ln n}{n}\binom{n}{\lfloor n/2\rfloor} \frac{5+o(1)}{2}\log n.$$
%The same applies if $r>n/2$, as $m_{n,r}=m_{n,n-r}$ due to matroid duality.
%As $\log \frac{\e2^n(n+1)}{k_{n,\lfloor n/2\rfloor}}\leq \frac{3}{2}\log(n)(1+o(1))$ by an application of \eqref{binomial_bound2},  we have
%$$\log m_{n,r}\leq \frac{4 \ln n}{n}\binom{n}{\lfloor n/2\rfloor} \frac{3}{2}\log(n)(1+o(1))$$
%As $m_n=\sum_r m_{n,r}$, we have
%$\log \log m_n \leq n-\frac{3}{2}\log n+ 2 \log\log n + O(1)$
%as required.
\end{proof}

Remark: The difference between this upper bound and the lower bound of Knuth is $2\log\log n +O(1)$.
We note that this approach cannot give a gap of $O(1)$.
 In particular, even an upper bound on the cardinality of a dominating set of the Johnson graph $J(n,r)$ that is closer to $\binom{n}{r}/(r(n-r)+1)$ (which is clearly the best possible) could improve this gap to $\log\log n +O(1)$ at best. We cannot expect to do better: indeed, if we consider the above proof applied to bound the number of sparse paving matroids, it is inherently as wasteful as using the trivial counting bound \eqref{eq:triv-count} (as a minimal cover of a sparse paving matroid just lists the non-bases).
We proceed by describing a better technique for bounding the number of sparse paving matroids.

\section{\label{s:encoding}A procedure for encoding vertex sets}
\subsection{The procedure}
We now describe a procedure given by Alon, Balogh, Morris and Samotij \cite{ABMS2012}, for which they refer to Kleitman and Winston \cite{KleitmanWinston1982} as the original source. They use the procedure to encode a stable set $I$ as a pair $(S, I\setminus S)$, such that $S \subseteq I$ and the number of possibilities for both $S$ and $I\setminus S$ can be controlled. We will use it for that purpose in this section as well, but to prepare for other uses in this paper we generalize the procedure so that it takes a general vertex set $K$ and produces a pair $(S, A)$, satisfying
\begin{equation}\label{eq:generalinvariant}S \subseteq K \subseteq S \cup N(S) \cup A.\end{equation}
We stress that the encoding is not one-to-one, and several sets $K$ may produce the same pair $(S,A)$.
We will later describe why such a pair  $(S,A)$ is useful.

Throughout this section, $G$ is a $d$-regular graph on $N$ vertices, with $d>0$, and the smallest eigenvalue of its adjacency matrix is $-\lambda$.
We denote $\alpha \eqdef \frac{\lambda}{d + \lambda}$.
For a subset $A \subseteq V$, let $G[A]$ denote the subgraph of $G$ induced by $A$. Let $e(A)$ denote the number of edges with both end points in $A$, i.e.\ the number of edges in $G[A]$.
Let us assume there is some fixed linear ordering $\le_V$ of the vertices of $G$ (say according to their indices $1,\ldots,N$). By the canonical ordering of $A \subseteq V$, we refer to the following procedure to order the set $A$ linearly. Let $v$ be the vertex with maximum degree in $G[A]$; if there are multiple such $v$, take the one that is smallest with respect to $\le_V$. Call $v$ the first vertex in the canonical ordering, and apply the procedure iteratively to $A\setminus\{v\}$.

The procedure to produce $(S,A)$ (see Figure \ref{alg:proc}) maintains two disjoint sets of vertices: $S$ for selected and $A$ for available. Initially, no vertices are selected ($S = \emptyset$) and all vertices are available ($A = V$). During the procedure, the set $S$ will expand and the set $A$ will shrink, until $|A|\leq \alpha N$. Throughout we will maintain \eqref{eq:generalinvariant} as an invariant.
%This will uniquely determine the choices that are made in the main loop of the procedure.

\begin{algorithm}[H]
	\SetKwBlock{Either}{either}{}
	\SetKwBlock{Or}{or}{}
	\SetKwInOut{Input}{Input}
	\SetKwInOut{Output}{Output}
	\Input{graph $G = (V,E)$, vertex set $K \subseteq V$}
	\Output{$(S,A)$}
	\BlankLine
	Set $A \leftarrow V$ and $S \leftarrow \emptyset$\;
	\While{$|A|> \alpha N$}
	{
		Pick the first vertex $v$ in the canonical ordering of $A$\;
		\uIf{$v \in K$}{set $S \leftarrow S \cup \{v\}$ and set $A \leftarrow A\setminus(N(v) \cup \{v\})$;}		
		\Else{set $A \leftarrow A\setminus\{v\}$;}

	}
	Output $(S, A)$.	
	%\caption{\label{alg:proc} Construction of $(S,A)$. }

\caption{\label{alg:proc}The encoding procedure.}
\end{algorithm}

\vspace{2mm}

The following is a simple but subtle and crucial observation from \cite{ABMS2012}.
\begin{lemma}\label{lemma:alon_observations}  Upon the termination of the algorithm, the set $A$ is completely determined by $S$ (irrespective of the set $K$).\end{lemma}
This follows as at any step in the algorithm, the vertices chosen thus far in $S$ completely determine the remaining vertices and their ordering.
In particular, given $S$ one can recover $A$ as follows: Initialize $X=V$ and $T=S$. Repeating the following steps until $|X| \leq  \alpha N$ (the resulting set $X$ when the algorithm terminates will be $A$).
(i) Consider the canonical ordering of $X$, and let $v$ be the first vertex in this ordering. (ii) If $v\in T$, discard $v$ from $T$ and $\{v\} \cup N(v)$ from $X$ and go to step 1. Otherwise, discard $v$ from $X$ and go to step 1.

%. the algorithm follows  that given the set $S$, one can uniquely determine which vertices were discarded until the algorithm terminated.
%
%The second statement in the lemma (which is the part that is of interest to us) follows from the first part, as the output $S$ determines the set $A' %\supset A$ immediately after the last vertex was added to $S$. The set $A$ is determined from $A'$ by iteratively removing the first vertex in the %canonical ordering until the stopping criterion is met.

\subsection{Application to counting stable  sets}
Later we will show:
\begin{lemma}\label{lemma:Sbound}The number of vertices selected into $S$ is at most $\lceil\frac{\ln(d+1)}{d+\lambda}N\rceil$.\end{lemma}

Let us first see how this implies the following upper bound on $i(G)$, the number of stable sets in $G$.
\repeattheorem{thm:indsets}
%\begin{theorem}\label{thm:indsets}
%	Let $G$ be a $d$-regular graph on $N$ vertices with smallest eigenvalue $-\lambda$. Then
%	$$i(G)\leq \sigma N{N \choose \sigma N}2^{\alpha N}$$
%where $\alpha= \frac{\lambda}{d + \lambda}$ and $\sigma=\frac{\ln(d+1)}{d+\lambda}$.
%\end{theorem}
\begin{proof}
	 Let $K$ be any stable set. Running the procedure yields a pair $S, A$ with $|A|\leq \alpha N$, such that
(i) $A$ is completely determined by $S$ (by Lemma \ref{lemma:alon_observations}) and (ii) $S \subseteq K \subseteq S \cup N(S)\cup A$.
Now, since $K$ is a stable set and $S \subseteq K$, we have $N(S)\cap K=\emptyset$. Together with (ii) above, this implies that $ K \subseteq S \cup A$.
Thus, $K = S \cup (K \cap A)$ and hence $K$ is completely determined by $S$ and $K \cap A$.

 As $A$ is completely determined by $S$, for a fixed $S$, there are at most $2^{\alpha N}$ possibilities for $K \cap A$. Moreover, as
 $|S|\leq \lceil \sigma N\rceil$, the number of ways of choosing $S$ is  at most $\sum_{i=0}^{\lceil\sigma N\rceil} \binom{N}{i}$.
\end{proof}

\subsection{Analysis}

We now prove Lemma \ref{lemma:Sbound}.
We first need the following lemma that was proved by Alon and Chung in \cite{AlonChung1988}, and earlier by Haemers (Theorem 2.1.4 (i) of  \cite{Haemers1980}). We use the version of the lemma stated in \cite{ABMS2012}.
\begin{lemma}\label{lemma:edgebound}
         For all $A\subseteq V(G)$, we have
	$
		2e(A) \ge |A| \left(\frac{d}{N} |A| - \lambda\frac{N-|A|}{N}\right).
	$
\end{lemma}
\proof
	Let $x$ denote the incidence vector of set $A$, and let $B$ denote the adjacency matrix of $G$. Then the number of edges is given by $(1/2) x^T Bx$. Let $v$ be the all $1$'s vector scaled by $|A|/N$, then $v \cdot (x-v)=0$, and $v$ is an eigenvector of $B$ with eigenvalue $d$. As $v$ is an eigenvector of $B$,  $v\perp (x-v)$ implies that $Bv\perp (x-v)$. Thus
	\begin{eqnarray*}
			2e(A)
				&= &  x^T B x =     (x-v+v)^T B (x-v+v) \\
				&= & (x-v)^TB(x-v) + v^T B v   \\
				&\geq &   -\lambda\|x-v\|^2 + d \|v\|^2 \\
				&= & -\lambda \left(\frac{(N-|A|)|A|^2}{N^2} + \frac{|A|(N-|A|)^2}{N^2}\right) + d\frac{|A|^2}{N} \\
				&= &  |A|\left(\frac{d|A|}{N} - \lambda \frac{(N-|A|)}{N}\right).
\qed
	\end{eqnarray*}

\begin{corollary}\label{cor:degree}
For any $\varepsilon > 0$, if $|A| = (\alpha + \varepsilon)N$, then $G[A]$ contains a vertex of degree at least $\varepsilon (d+\lambda)$.\end{corollary}
\begin{proof}
	Let $A$ be a vertex set of size $(\alpha + \varepsilon)N$. By Lemma \ref{lemma:edgebound}, $2e(A) \ge |A|(d+\lambda)\varepsilon$. Hence the average degree in $G[A]$ is at least $\varepsilon (d+\lambda)$, and so there must be some vertex in $G[A]$ of degree $\ge \varepsilon (d+\lambda)$.
\end{proof}
In particular, it follows that a stable set $A\subseteq V(G)$ has size at most $\alpha N$ (this is Hoffmann's bound).

%We now prove Lemma \ref{lemma:Sbound}.
%\begin{lemma}\label{lemma:Sbound}The number of vertices selected into $S$ is at most $\frac{\ln(d+1)}{d+\lambda}N$.\end{lemma}
\begin{proof}[Proof of Lemma \ref{lemma:Sbound}]
	Say that the procedure is in phase $j$ if the current $A$ satisfies
	$$\frac{|A|}{N} \in \left(\alpha + \frac{j-1}{d+\lambda}, \alpha + \frac{j}{d+\lambda}\right], \qquad\qquad (j = d, d-1, \ldots, 1).$$
	Then each phase sees the removal of at most $N/(d+\lambda)$ vertices from $A$. By Corollary \ref{cor:degree}, any vertex that gets selected into $S$ during phase $j$ has degree $>j-1$, hence removes at least $j+1$ vertices from $A$ (the vertex selected into $S$, and its neighbors). Let $S(j)$ be the set of vertices that get selected into $S$ during phase $j$, then the above argument shows that $|S(j)| \le \lceil\frac{N}{(d+\lambda)(j+1)}\rceil$. Summing up we obtain
	$$|S| = \sum_{j=1}^d |S(j)| \le \sum_{j=1}^d \left\lceil\frac{N}{(d+\lambda)(j+1)}\right\rceil<\frac{\ln(d+1)}{d+\lambda}N+d,$$
as $\sum_{j=1}^d \frac{1}{j+1}\leq \ln (d+1) $. 	

To obtain the slightly sharper bound stated in the lemma, we make a more refined analysis. Let $u_j$ be the fractional number of vertices that get removed in phase $j$ due to the insertion of a vertex in $S(j+1)$ for $j=d,\ldots 0$. That is, if $v$ is the last vertex which is inserted in $S(j+1)$ during phase $j+1$, then $u_j:=\max\{0,(\alpha + \frac{j}{d+\lambda})N-|A|\}$ just after removing $N(v)\cup\{v\}$ from $A$ in the procedure.  
	Then $|S(j)| \le \frac{N}{(d+\lambda)(j+1)}+\frac{u_{j-1}-u_{j}}{j+1}$, since inserting vertices in $S(j)$ removes vertices from $A$ only while $$|A| \in \left(\alpha N + \frac{j-1}{d+\lambda}N-u_{j-1}, \alpha N + \frac{j}{d+\lambda}N-u_{j}\right].$$
	 It follows that
	$$|S| = \sum_{j=1}^d |S(j)| \le \frac{N}{d+\lambda} \sum_{j=1}^d \frac{1}{j+1} +\sum_{j=1}^d \frac{u_{j-1}-u_{j}}{j+1}<\frac{\ln(d+1)}{d+\lambda}N+1,$$
as $0 \leq u_j< j+2$, $u_d=0$, and hence $\sum_{j=1}^d \frac{u_{j-1}-u_{j}}{j+1}\leq \frac{u_0}{2}<1$. This proves the lemma.
\end{proof}

%Theorem \ref{thm:indsets}, which we state here again for convenience.

\section{\label{s:sparsepaving}An upper bound on the number of sparse paving matroids}

As shown in Lemma \ref{lemma:sparsepaving_johnson}, sparse paving matroids of rank $r$ on groundset $[n]$ correspond one-to-one to stable sets in the Johnson graph $J(n,r)$. Thus
$s_{n,r}=i(J(n,r)),$
and we may apply Theorem \ref{thm:indsets} to bound the number of sparse paving matroids. We first investigate the parameters that occur in this application of the theorem.

Recall  that $J(n,r)$ is regular of degree $r(n-r)$ and has $\binom{n}{r}$ vertices.
The eigenvalues of the adjancency matrix of $J(n,r)$  are $r(n-r) - i(n+1-i)$ for $i = 0, 1, \ldots, r$ if $r\leq n/2$ (see \cite{BCNbook}).
The minimal eigenvalue is attained for $i=r$, which identifies the smallest eigenvalue as $-\lambda_{n,r}$, where $\lambda_{n,r}=r$.
%\begin{equation}
%		\lambda_{n,r} =
%		\begin{cases}
%			r & \text{if $r\leq n/2$,} \\
%			n-r & \text{if $r\geq n/2$.}
%		\end{cases}
%	\end{equation}
For $r, n$ such that $0<r\leq n/2$, we define 
$$\alpha_{n,r}\eqdef \frac{\lambda_{n,r}}{r(n-r) + \lambda_{n,r}}=\frac{1}{n-r+1}\text{ and } \sigma_{n,r}\eqdef \frac{\ln( r(n-r)+1)}{r(n-r)+\lambda_{n,r}}=\frac{\ln(r(n-r)+1)}{r(n-r+1)}$$
\begin{lemma}\label{alpha_bound}If $n\geq 4$ and $0< r \leq n/2$, then $\alpha_{n,r}\binom{n}{r}\leq \frac{2}{n}\binom{n}{\lfloor n/2\rfloor}$ and $\lceil \sigma_{n,r}\binom{n}{r}\rceil\leq \frac{8\ln (n^2)}{n^2}\binom{n}{\lfloor n/2\rfloor}$. \end{lemma}
\proof Clearly $\alpha_{n,r}=1/(n-r+1)\leq 2/n$ whenever $0<r\leq n/2$. If $n\geq 4$, we have
$$\sigma_{n,r}\binom{n}{r}\leq f(n,r):= \frac{2\ln( n^2)}{(r+1)(n-r+1)}\binom{n}{r}.$$
Since $f(n,r+1)/f(n,r)=(n-r+1)/(r+2)\geq 1$ while $r\leq n/2-1$, it follows that
$$\left\lceil\sigma_{n,r}\binom{n}{r}\right\rceil\leq \lceil f(n,\lfloor n/2\rfloor)\rceil\leq \left\lceil\frac{1}{1+2/n}\frac{8\ln(n^2)}{n^2}\binom{n}{\lfloor n/2\rfloor}\right\rceil\leq \frac{8\ln(n^2)}{n^2}\binom{n}{\lfloor n/2\rfloor}$$
whenever $0<r\leq n/2$ and $n\geq 4$, where we use that $\frac{1}{1+2/n}\leq 1- 2/n$.
\endproof

This gives sufficient control over the parameters to prove Theorem~\ref{thm:sn} from Theorem~\ref{thm:indsets}.
%\begin{theorem}$\log \log s_n \le n - \frac{3}{2}\log n + \frac{1}{2}\log\frac{2}{\pi} + 1 + o(1)$.\end{theorem}
\repeattheorem{thm:sn}
\proof Let $n>2$ and $0<r\leq n/2$. Then $\lceil\sigma_{n,r} N\rceil< N/2$, and using Theorem  \ref{thm:indsets} we have
$$s_{n,r}=i(J(n,r))\leq \lceil\sigma_{n,r}N\rceil{N \choose \lceil\sigma_{n,r} N\rceil}2^{\alpha_{n,r} N},$$
where $N=\binom{n}{r}$. 
Using \eqref{binomial_bound} to bound the factor ${N \choose \lceil\sigma_{n,r} N\rceil}$ and applying logarithms, we obtain 
$$\log s_{n,r}\leq \log N + \lceil\sigma_{n,r}N\rceil\log (\frac{\e N}{\lceil\sigma_{n,r}N\rceil}) + \alpha_{n,r} N.$$
Applying Lemma \ref{alpha_bound} and using that $N\leq \binom{n}{\lfloor n/2\rfloor}\leq 2^n$, we obtain
$$\log s_{n,r}\leq n + \frac{8\ln (n^2)}{n^2}\binom{n}{\lfloor n/2\rfloor}\log (n^2) + \frac{2}{n}\binom{n}{\lfloor n/2\rfloor}$$
for sufficiently large $n$  and $0< r\leq n/2$. As $\log(n)^2/n^2\leq o(1/n)$, the bound on $\alpha_{n,r}N$ dominates as $n\rightarrow\infty$. Using that $s_{n,0}=1$ and $s_{n,r}=s_{n,n-r}$, we have 
$$\max_{0\leq r\leq n}\log s_{n,r}\leq \frac{2}{n}\binom{n}{\lfloor n/2\rfloor}(1+o(1))\text{ as } n\rightarrow\infty.$$
As $s_n=\sum_{r=0}^n s_{n,r}\leq (n+1)\max_r s_{n,r}$, we also have $\log s_n\leq \frac{2}{n}\binom{n}{\lfloor n/2\rfloor}(1+o(1))$. 
Taking logarithms and applying \eqref{binomial_bound2} to bound the binomial coefficient, the result follows.
\endproof

\section{\label{s:bound}An upper bound on the number of matroids}

We will now show the upper bound on $m_n$ claimed in Theorem \ref{thm:mn}.
To do this, we first show that substantially better local covers at $X$ exist if $X$ is a non-basis. Later, we combine this fact with the encoding
procedure in Section \ref{s:encoding} to find a very concise encoding of a matroid $M\in\MM_{n,r}$.

\subsection{Improved Local Covers}
\begin{lemma} \label{local_cover2} Let $M\in\MM_{n,r}$. For each  $r$-set $X \in  \binom{E}{r}$ that is dependent in $M$, there exists a
%local cover $\mathcal{Z}_X$ such that
set  $\mathcal{Z}_X\subseteq \FF(M)$ such that $|\mathcal{Z}_X| \leq 2$, and each non-basis $Y\in N(X)\cup\{X\}$ is covered by some $F\in \mathcal{Z}_X$.
\end{lemma}
\proof
Let $X$ be a fixed $r$-set.  If  $r_M(X)<r-1$, take $\mathcal{Z}_X\eqdef\{\cl_M(X)\}.$ Then if  $Y \in N(X)$ or $Y=X$, we have
$$|\cl_M(X)\cap Y|\geq |X\cap Y|\geq r-1> r_M(X)=r_M(\cl_M(X)).$$
If $r_M(X)=r-1$, then $X$ contains a unique circuit $C$ of $M$. Take $\mathcal{Z}_X\eqdef\{\cl_M(C), \cl_M(X)\}.$ If  $Y \in N(X)$ is not a basis, then by submodularity
$$r_M(X\cup Y)+r_M(X\cap Y)\leq r_M(X)+r_M(Y)<2r-1$$
so that  $r_M(X\cup Y)<r$ or $r_M(X\cap Y)<r-1$. In the former case, we have $Y\subseteq \cl_M(X)$, hence
$$|\cl_M(X)\cap Y|=r>r_M(X\cup Y)\geq r_M(X)=r_M(\cl(X)).$$
In the latter case, $X\cap Y$ is dependent and hence must contain a circuit $C'$, and as $C$ is the unique circuit contained in $X$, we must have $C'=C$. Then
$|\cl_M(C)\cap Y|\geq |C|>r_M(C)=r_M(\cl_M(C)).$
\endproof

\subsection{Matroid Encoding}
The crucial difference from Lemma \ref{local_cover} is the assumption that $X$ is a dependent set in $M$. This allows us to obtain a much smaller bound on the size of a cover of $M$, if we can identify a small collection of non-bases of $M$ such that their neighborhood contains all non-bases in a large fraction of the $r$-sets. This is exactly what the encoding algorithm will accomplish.
 %Roughly speaking, either a matroid has very few non-bases, in which case we can already bound the number of such matroids. Otherwise, if a matroid has too many non-bases,
%then we can find a small collection of non-bases such that their neighborhood contains most of the $r$-sets, and obtain a concise representation using the local covers at these non-bases.
We now give the details.

%by adapting the procedure for counting independent sets in the Johnson graph to find a very concise encoding of a matroid $M\in\MM_{n,r}$. We will apply the encoding algorithm to the set $K$ of nonbases of a matroid $M$. This will yield sets $S,A$ such that
%$$S\subseteq K\subseteq S\cup N(S)\cup A$$
%where the sizes of $S$ and $A$ are controlled, and $S$ determines $A$. As in Theorem \ref{thm:indsets}, this will put an adequate bound on the number of possibilities for $K\cap A$. To bound the number of possibilities for $K\cap (S+N(S))$, we will show that there exists a set of flats $\ZZ\subseteq \FF(M)$ that covers all the nonbases in $S+N(S)$, with $|\ZZ|\leq 2|S|$. For proving the latter bound, we make the following adaptation of Lemma \ref{local_cover} that gives a sharper bound on the size of a local cover for {\em dependent} sets of the matroid.

%\begin{theorem}  $\log\log m_n\leq n-\frac{3}{2}\log n +\frac{1}{2}\log\frac{2}{\pi} +1 +o(1)$.\end{theorem}
\repeattheorem{thm:mn}
\proof Let $n,r$ be such that $0<r<n/2$. Consider a matroid $M\in\MM_{n,r}$, and let  $K\eqdef\binom{E}{r}\setminus \BB(M)$ be the set of its non-bases.
Then $K$ is a set of vertices of the graph $G=J(n,r)$. As before, let $N=\binom{n}{r}$ denote the number of vertices of $G$, let $d=r(n-r)$ be its degree and $-\lambda=-r$ be the smallest eigenvalue of $G$. Let $\alpha=\frac{\lambda}{d+\lambda} = 1/(n-r+1)$ and $\sigma=\frac{\ln (d+1)}{d+\lambda}$.
We describe how to obtain a concise description of $M$.

Apply the encoding procedure to $K$ and obtain sets  $S,A$ such that  $|A|\leq \alpha N$, $|S|\leq \lceil\sigma N\rceil$, $A$ is determined by $S$, and $$S\subseteq K\subseteq S\cup N(S) \cup A.$$
By Lemma \ref{local_cover2}, there exists a local cover $\ZZ_X$ with $|\ZZ_X|\leq 2$ for each $X\in S$, where we use that each such $X$ is a dependent set of $M$. Then $$\ZZ\eqdef\bigcup_{X\in S}\mathcal{Z}_X$$ covers all $Y\in (S\cup N(S))\setminus\BB$, and $|\ZZ|\leq 2|S|$. As all members of $K\setminus A$ lie in $S\cup N(S)$, the set $K\setminus A$ is fully determined by  $\{(F, r_M(F))\mid F\in \ZZ\}$.
For the remaining non-bases in $K\cap A$, we can simply list them.
Thus, $(\{(F, r_M(F))\mid F\in \ZZ\}, K \cap A)$ gives a complete and concise description of the non-bases of the matroid $M$.

This bounds the number of matroids in $\MM_{n,r}$ by the number of ways of choosing  $S$ from an $N$-set, times the number of ways of choosing the collection $\{(F, r_M(F))\mid F\in \ZZ\}$ from a set of size $2^n(n+1)$, times the number of possible subsets from $A$. As $|A|\leq \alpha N$, $|S|\leq \lceil\sigma N\rceil$ and $|\ZZ|\leq 2|S|$, this yields 
$$m_{n,r}\leq \lceil\sigma N\rceil\binom{N}{\lceil\sigma N\rceil}2\lceil\sigma N\rceil\binom{2^n(n+1)}{2\lceil\sigma N\rceil}2^{\alpha N}$$
for $n$ sufficiently large so that $2\lceil\sigma N\rceil<N/2$ for all $r$.
Using \eqref{binomial_bound} to bound the factors $\binom{N}{\lceil\sigma N\rceil}$ and $\binom{2^n(n+1)}{2\lceil\sigma N\rceil}$, and applying logarithms, we obtain  
$$\log m_{n,r}\leq \log(\lceil\sigma N\rceil)+   \lceil\sigma N\rceil\log\frac{\e N}{\lceil\sigma N\rceil}  + \log(2\lceil\sigma N\rceil)+2\lceil\sigma N\rceil\log\frac{\e 2^n(n+1)}{2\lceil\sigma N\rceil}+\alpha N $$
By Lemma \ref{alpha_bound}, we have $\lceil\sigma N\rceil \leq \frac{8 \ln n}{n^2}\binom{n}{\lfloor n/2\rfloor}$ and  $\alpha N\leq  \frac{2}{n}\binom{n}{\lfloor n/2\rfloor}$ for sufficiently large $n$, and we bound $N\leq \binom{n}{\lfloor n/2\rfloor}\leq 2^n$ to obtain
$$\log m_{n,r}\leq n + \frac{8\ln (n^2)}{n^2}\binom{n}{\lfloor n/2\rfloor}\log (n^2) + (n+1)+ 2\frac{8\ln (n^2)}{n^2}\binom{n}{\lfloor n/2\rfloor}\log (n^4) + \frac{2}{n}\binom{n}{\lfloor n/2\rfloor}$$
for $0<r<n/2$ and sufficiently large $n$. As $\log(n)^2/n^2\leq o(1/n)$, the bound on $\alpha N$ dominates as $n\rightarrow \infty$. Using that $m_{n,0}=1$ and $m_{n,r}=m_{n,n-r}$ we conclude that 
$$\max_{0\leq r\leq n} \log m_{n,r}\leq \frac{2}{n}\binom{n}{\lfloor n/2\rfloor}(1+o(1)) \text{ as }n\rightarrow\infty.$$
 Since $m_n=\sum_{r=0}^nm_{n,r}\leq (n+1)\max_r m_{n,r}$, we also have
\begin{equation}\label{eq:log_mn}
\log m_n\leq \frac{2}{n}\binom{n}{\lfloor n/2\rfloor}(1+o(1)).
\end{equation}
Taking logarithms and applying \eqref{binomial_bound2} to bound the binomial coefficient, the theorem follows.
\endproof

Combining the above theorem with Knuth's lower bound on $s_n$ \eqref{eq:knuthlowerbound}, we obtain:
%\begin{corollary} \label{relative_bound}$\log\log m_n\leq \log\log s_n +1 +o(1).$
%\end{corollary}
\repeatcorollary{relative_bound}

\section{Further directions}
\label{s:outtro}
\subsection{The maximal stable sets of the Johnson graph}
By Knuth's lower bound and our result,
$$\frac{1}{n}\binom{n}{\lfloor n/2\rfloor} \leq \log s_n \leq \log m_n \leq \frac{2}{n} \binom{n}{\lfloor n/2\rfloor}(1+o(1)).$$
So asymptotically, there is a factor 2 gap between the lower and upper bounds on $\log m_n$ (which gives the +1 gap in Corollary \ref{relative_bound}).
Any improvement of this gap, would immediately reduce the gap between the  
known lower and upper bounds on the size of the maximum stable set in $J(n,n/2)$ (this question has been studied in coding, see e.g.~\cite{BE11} and references therein).

%The lower bound is the size of a stable set in $J(n, \lfloor n/2\rfloor)$ as constructed by Graham and Sloane \cite{GrahamSloane1980}. As far as we know, the best general upper bound on the size of such a stable set is about $2\binom{n}{\lfloor n/2\rfloor}/n$, as a consequence of Hoffman's bound (corollary \ref{cor:degree}).

Conversely, it is possible that a better understanding of the maximum size of a stable set in $J(n,r)$ could lead to better bounds for $m_n$.
If it could be shown that $J(n, \lfloor n/2\rfloor)$ actually has a stable set of size $\approx 2\binom{n}{\lfloor n/2\rfloor}/n$, then the gap of +1 in Corollary \ref{relative_bound} would reduce to $o(1)$. On the other hand, if the maximum size of a stable set is at most $\approx \binom{n}{\lfloor n/2\rfloor}/n$, then the technique for showing such an upper bound could potentially be useful for bounding $m_n$.

%\subsection{The structure of stable sets of the Johnson graph}
%The upper bound on the number of stable sets in the Johnson graph $J(n,r)$ shows that among the sets of vertices of size at most $\alpha(J(n,r))$ there are relatively few stable sets. We ask if there could be a structural reason for this, that is, a structure common to all stable sets from which it directly follows that there cannot be many independent sets.
%
%For example, if  there exists a `small' set $U$ of maximal stable sets of $J(n,r)$ and a `small' $k\in\N$, such that for each stable set $I$ of $J(n,r)$ there is a $J\in U$ such that $|I\setminus J|\leq k$, then it would follow that there are at most $|U|\binom{N}{k}2^{\alpha N}$ stable sets, where $N=\binom{n}{r}$ and $\alpha N$ bounds the cardinality of a stable set of $J(n,r)$.

\subsection{The number of matroids without circuit-hyperplanes}
If $M=(E,\BB)$ is a matroid of rank $r$, then the following are equivalent for a subset $X\subseteq E$:
\begin{itemize}
\item $X$ is both a circuit and a hyperplane  of $M$; and
\item $X$ is an isolated vertex of $J(E,r)\setminus \BB$.
\end{itemize}
It is well-known that if $X$ is a circuit-hyperplane of $M$, then {\em relaxing} $X$ gives another matroid $(E, \BB\cup\{X\})$. More strongly, we have:
\begin{lemma}Let $0<r<|E|$, let $\BB\subseteq \binom{E}{r}$, and let $U$ be a set of isolated vertices of $J(E,r)\setminus\BB$. Then
$$(E,\BB)\text{ is a matroid }\Longleftrightarrow (E, \BB\cup U)\text{ is a matroid}.$$
\end{lemma}
By the lemma, each matroid $M=(E,\BB)$ can be decomposed uniquely  as $(\tilde{M}, U)$, where $\tilde{M}:=(E, \BB\cup U)$ and  $U$ is the set of all isolated vertices of $J(E,r)\setminus\BB$. Then $\tilde{M}$ has no circuit-hyperplanes, and $U$ is a stable set of the Johnson graph. Hence,
$$m_{n,r}/ i(J(n,r))\leq  |\{M\in\MM_{n,r}\mid M\text{ has no circuit-hyperplanes }\}|. $$
It follows that any upper bound on the number of circuit-hyperplane-free matroids will yield an upper bound on all matroids, using Theorem
\ref{thm:sn} to bound $i(J(n,r))=s_{n,r}$. We think that it may be possible to directly prove Theorem \ref{thm:mn} along these lines, and  conjecture that the number of circuit-hyperplane-free matroids is relatively small.
\begin{conjecture}$$\lim_{n\rightarrow \infty} \frac{|\{M\in\MM_{n}\mid M\text{ has no circuit-hyperplanes }\}|}{m_n}=0.$$
\end{conjecture}
This would also follow from Conjecture \ref{MNWW}, as the only sparse paving matroids without circuit-hyperplanes are the uniform matroids.

Mayhew, Newman and Whittle have recently shown that real-representability of matroids cannot be captured by a `natural' axiom, that is, a sentence of {\em monadic second-order  logic for matroids}, which is defined in their paper \cite{MayhewNewmanWhittle2012}.
We believe that the difficulty may be in the seemingly unmanageable set of stable sets of the Johnson graph,
and conjecture that there is not even a natural axiom that captures whether a {\em sparse paving} matroid is real-representable. 

%On the other hand, we wonder whether this difficulty can be factored out, and ask if there is a natural axiom that describes which circuit-hyperplane-free, connected matroids are real-representable.

\subsection{The cover complexity of a matroid}
%For a matroid $M$, we define the {\em cover complexity}  as $$\kappa(M)\eqdef\min\{|\mathcal{Z}| \mid \mathcal{Z}\subseteq\FF(M), \mathcal{Z}\text{ is a flat cover of }M\}.$$
%The following lemma is straightforward.
%\begin{lemma} Let $M$ be a matroid. Then
%\begin{enumerate}
%\item $\kappa(M)=\kappa(M^*)$;
%\item if $e$ is neither a loop nor a coloop of $M$, then $\kappa(M)\leq \kappa(M/e)+\kappa(M\setminus e)$;
%\item if $N$ is a minor of $M$, then $\kappa(M)\geq \kappa(N)$;
%\item if $N$ arises from $M$ by relaxing a circuit-hyperplane, then $\kappa(M)=\kappa(N)+1$.
%\end{enumerate}
%\end{lemma}
%In \cite[Conj. 1.7]{MayhewNewmanWelshWhittle2011} it is conjectured that if $N$ is any sparse paving matroid, then
%$$\lim_{n\rightarrow \infty} \frac{|\{M\in\mathbb{M}_n\mid M\text{ does not have an }N\text{-minor}\}|}{m_n}=0.$$
%In a forthcoming paper, we will show that the conjecture holds for $N=U_{2,k}$ and $N=U_{3,6}$, by deriving bounds on the cover complexity of matroids not having such a minor $N$.
%We pose the challenge of bounding $$\max\{\kappa(M)\mid M\in \mathbb{M}_n, M\text{ does not have an }M(K_4)\text{-minor }\}.$$
Inspired by the flat covers from Section~\ref{s:bound-weak}, we define the {\em cover complexity} of a matroid $M$ as
$$\kappa(M)\eqdef\min\{|\mathcal{Z}| \mid \mathcal{Z}\subseteq\FF(M), \mathcal{Z}\text{ is a flat cover of }M\},$$
i.e.\ the minimal size of such a flat cover. The cover complexity has a number of properties that one would expect from a complexity measure on matroids; e.g.\ ~it is invariant under taking duals, and monotone under taking minors.

In \cite[Conj.\ ~1.7]{MayhewNewmanWelshWhittle2011} it is conjectured that if $N$ is any sparse paving matroid, then
$$\lim_{n\rightarrow \infty} \frac{|\{M\in\mathbb{M}_n\mid M\text{ does not have an }N\text{-minor}\}|}{m_n}=0.$$
In \cite{PendavingVanderPol2013} two of the current authors show that this conjecture holds for each of the matroids $N = U_{2,k}$, $U_{3,6}$, $P_6$, $Q_6$, and $R_6$, by deriving bounds on the cover complexity of matroids that do not have such a minor.

\section{Acknowledgements}
We thank Andries Brouwer for several useful comments and Dominic Welsh for his help in tracing the origins of the conjecture that `most matroids are paving'.
\bibliographystyle{plain}
\bibliography{math}

\begin{thebibliography}{10}

\bibitem{AlonChung1988}
N.~Alon and F.~R.~K. Chung.
\newblock Explicit construction of linear sized tolerant networks.
\newblock {\em Discrete Math.}, 72(1-3):15--19, 1988.

\bibitem{ABMS2012}
Noga Alon, J\'{o}zsef Balogh, Robert Morris, and Wojciech Samotij.
\newblock Counting sum-free sets in {Abelian} groups.
\newblock arXiv:1201.6654 (to appear in Israel J. Math.), 2012.

\bibitem{AlonSpencerBook}
Noga Alon and Joel~H. Spencer.
\newblock {\em The probabilistic method}.
\newblock Wiley-Interscience Series in Discrete Mathematics and Optimization.
  John Wiley \& Sons Inc., Hoboken, NJ, third edition, 2008.
\newblock With an appendix on the life and work of Paul Erd{\H{o}}s.

\bibitem{BlackburnCrapoHiggs1973}
John~E. Blackburn, Henry~H. Crapo, and Denis~A. Higgs.
\newblock A catalogue of combinatorial geometries.
\newblock {\em Math. Comp.}, 27:155--166; addendum, ibid. 27 (1973), no. 121,
  loose microfiche suppl. A12--G12, 1973.

\bibitem{Bonin2011}
Joseph~E. Bonin.
\newblock Sparse paving matroids, basis-exchange properties, and cyclic flats.
\newblock arXiv:1011.1010v1, 2011.

\bibitem{BE11}
A.~E. Brouwer and T.~Etzion.
\newblock Some new constant weight codes.
\newblock {\em Advances in Mathematics of Communications}, 5:417--424, 2011.

\bibitem{BCNbook}
Andries~E. Brouwer, Arjeh~M. Cohen, and Arnold Neumaier.
\newblock {\em Distance-regular graphs}.
\newblock Springer, 1989.

\bibitem{BrouwerHaemers2012}
Andries~E. Brouwer and Willem~H. Haemers.
\newblock {\em Spectra of graphs}.
\newblock Universitext. Springer, New York, 2012.

\bibitem{CrapoRotaBook}
Henry~H. Crapo and Gian-Carlo Rota.
\newblock {\em On the foundations of combinatorial theory: {C}ombinatorial
  geometries}.
\newblock The M.I.T. Press, Cambridge, Mass.-London, preliminary edition, 1970.

\bibitem{GeelenHumphries2006}
Jim Geelen and Peter~J. Humphries.
\newblock Rota's basis conjecture for paving matroids.
\newblock {\em SIAM J. Discrete Math.}, 20(4):1042--1045 (electronic), 2006.

\bibitem{GrahamSloane1980}
R.~L. Graham and N.~J.~A. Sloane.
\newblock Lower bounds for constant weight codes.
\newblock {\em IEEE Trans. Inform. Theory}, 26(1):37--43, 1980.

\bibitem{Haemers1980}
Wilhelmus~Hubertus Haemers.
\newblock {\em Eigenvalue techniques in design and graph theory}, volume 121 of
  {\em Mathematical Centre Tracts}.
\newblock Mathematisch Centrum, Amsterdam, 1980.
\newblock Dissertation, Technische Hogeschool Eindhoven, Eindhoven, 1979.

\bibitem{Jerrum2006}
Mark Jerrum.
\newblock Two remarks concerning balanced matroids.
\newblock {\em Combinatorica}, 26(6):733--742, 2006.

\bibitem{KleitmanWinston1982}
Daniel~J. Kleitman and Kenneth~J. Winston.
\newblock On the number of graphs without {$4$}-cycles.
\newblock {\em Discrete Math.}, 41(2):167--172, 1982.

\bibitem{Knuth1974}
Donald~E. Knuth.
\newblock The asymptotic number of geometries.
\newblock {\em J. Combinatorial Theory Ser. A}, 16:398--400, 1974.

\bibitem{Kung1996}
Joseph P.~S. Kung.
\newblock Matroids.
\newblock In {\em Handbook of algebra, {V}ol.\ 1}, pages 157--184.
  North-Holland, Amsterdam, 1996.

\bibitem{MayhewNewmanWelshWhittle2011}
Dillon Mayhew, Mike Newman, Dominic Welsh, and Geoff Whittle.
\newblock On the asymptotic proportion of connected matroids.
\newblock {\em European J. Combin.}, 32(6):882--890, 2011.

\bibitem{MayhewNewmanWhittle2012}
Dillon Mayhew, Mike Newman, and Geoff Whittle.
\newblock Is the missing axiom of matroid theory lost forever?
\newblock arXiv:1204.3365, 2012.

\bibitem{MayhewRoyle2008}
Dillon Mayhew and Gordon~F. Royle.
\newblock Matroids with nine elements.
\newblock {\em J. Combin. Theory Ser. B}, 98(2):415--431, 2008.

\bibitem{MayhewWelsh2010}
Dillon Mayhew and Dominic Welsh.
\newblock On the number of sparse paving matroids.
\newblock http://homepages.ecs.vuw.ac.nz/~mayhew/Publications/MW.pdf, 2010.

\bibitem{MerinoNoble2010}
Criel Merino, Steven~D. Noble, Marcelino Ram\'irez-Ib\'a\~nez, and Rafael
  Villarroel-Flores.
\newblock On the structure of the h-vector of a paving matroid.
\newblock {\em Eur. J. Comb.}, 33(8):1787--1799, 2012.

\bibitem{OxleyBook}
James Oxley.
\newblock {\em Matroid theory}, volume~21 of {\em Oxford Graduate Texts in
  Mathematics}.
\newblock Oxford University Press, Oxford, second edition, 2011.

\bibitem{PendavingVanderPol2013}
R.A. Pendavingh and J.G. {van der}~Pol.
\newblock Counting matroids in minor-closed classes.
\newblock arXiv:1302.1315v3, 2013.

\bibitem{Piff1973}
M.~J. Piff.
\newblock An upper bound for the number of matroids.
\newblock {\em J. Combinatorial Theory Ser. B}, 14:241--245, 1973.

\bibitem{PiffWelsh1971}
M.~J. Piff and D.~J.~A. Welsh.
\newblock The number of combinatorial geometries.
\newblock {\em Bull. London Math. Soc.}, 3:55--56, 1971.

\bibitem{SchrijverBookB}
Alexander Schrijver.
\newblock {\em Combinatorial optimization. {P}olyhedra and efficiency. {V}ol.
  {B}}, volume~24 of {\em Algorithms and Combinatorics}.
\newblock Springer-Verlag, Berlin, 2003.
\newblock Matroids, trees, stable sets, Chapters 39--69.

\bibitem{WelshBook}
D.~J.~A. Welsh.
\newblock {\em Matroid theory}.
\newblock Academic Press [Harcourt Brace Jovanovich Publishers], London, 1976.
\newblock L. M. S. Monographs, No. 8.

\bibitem{Whitney1935}
Hassler Whitney.
\newblock On the {A}bstract {P}roperties of {L}inear {D}ependence.
\newblock {\em Amer. J. Math.}, 57(3):509--533, 1935.

\end{thebibliography}
\end{document}